\documentclass{amsart}
\usepackage{amssymb,amsthm,latexsym,amsfonts}
\usepackage{palatino,fancyhdr,graphicx,hyperref}
\usepackage{verbatim}
\usepackage{tikz-cd}
\usepackage{enumitem}
\usepackage{setspace}
\usepackage{amsmath}
\usepackage{adjustbox} 

\hypersetup{
colorlinks,
citecolor=blue,
pdfauthor={Hou Xiaolu},
pdftitle={Paper}}

\def \RR {{\mathbb R}}
\def \QQ {{\mathbb Q}}

\def \ZZ {{\mathbb Z}}

\def \a {{\mathfrak a}}

\def \p {{\mathfrak p}}
\def \Pf {{\mathfrak P}}

\def \Oc {{\mathcal O}}
\def \D {{\mathcal D}}

\def \re {{{\rm Re}}}
\def \im {{{\rm Im}}}

\newcommand\Tr[1] {{\rm{Tr}}\left(#1\right)}

\newcommand\md[2] {\equiv#1~{\rm mod}~#2}

\theoremstyle{definition}
\newtheorem{thm}{Theorem}[section]
\newtheorem{lem}[thm]{Lemma}

\newtheorem{definition}[thm]{Definition}

\newtheorem{lemma}[thm]{Lemma}
\newtheorem{proposition}[thm]{Proposition}

\newtheorem{corollary}[thm]{Corollary}
\newtheorem{remark}[thm]{Remark}
\newtheorem{example}[thm]{Example}

\title{Construction of Arakelov-modular Lattices from Number Fields}
\author{Xiaolu Hou}
\address{
Division of Mathematical Sciences\\
Nanyang Technological University, Singapore}
\email {HO0001LU@e.ntu.edu.sg}

\begin{document}

\begin{abstract}
An Arakelov-modular lattice of level $\ell$, where $\ell\in\ZZ_{>0}$, is an $\ell-$modular lattice constructed from a fractional ideal of a CM field such that the lattice can be obtained from its dual by multiplication of an element with norm $\ell$.
The characterization of existence of Arakelov-modular lattices has been completed for cyclotomic fields~\cite{Bayer}. 
In this paper, we extend the definition to totally real number fields and study the criteria for the existence of Arakelov-modular lattices over totally real number fields and CM fields.
We give the characterization of Arakelov-modular lattices over the maximal real subfield of a cyclotomic field with prime power degree and totally real Galois fields with odd degrees.
Characterizations of Arakelov-modular lattices of trace type, which are special cases of Arakelov-modular lattices, are given for quadratic fields and maximal real subfields of cyclotomic fields with non-prime power degrees.
\end{abstract}

%
%
\maketitle

\vspace{1cm}
\\


%
%

\section{Introduction}

A lattice is a pair $(L,b)$, where $L$ is a free $\ZZ-$module and $b:L\otimes_\ZZ\RR\times L\otimes_\ZZ\RR\to\RR$ is a positive definite symmetric $\ZZ-$bilinear form~\cite{Ebeling,Martinet}.
The dual of a lattice $(L,b)$ is the lattice $(L^*,b)$, where
\[
	L^*=\{x\in L\otimes_\ZZ\RR:b(x,y)\in\ZZ\ \forall y\in L\}.
\]
A lattice is called \emph{integral} if $L\subseteq L^*$.
An integral lattice is \emph{even} if $b(x,x)\in2\ZZ\ \forall x\in L$ and \emph{odd} otherwise.
For a positive integer $\ell$, a lattice $(L,b)$ is said to be $\ell-$modular (or modular of level $\ell$) if $(L,b)$ is isomorphic to $(L^*,\ell b)$, i.e., if there exists a $\ZZ-$module homomorphism $\varphi:L^*\to L$ such that $\ell b(x,y)=b(\varphi(x),\varphi(y))$ for all $x,y\in L^*$~\cite{Quebbemann}.

In~\cite{Bayer} the definition of \emph{Arakelov-modular lattice} was introduced, which is a lattice constructed over a fractional ideal of a CM field such that it is isomorphic to its dual in a specific way.
In particular, an Arakelov-modular lattice of level $\ell$ is a modular lattice of level $\ell$.
The characterization of existence of Arakelov-modular lattices over cyclotomic fields was completed in \cite{Bayer}: given any cyclotomic field $K$, the possible values of $\ell$ such that there exists an Arakelov-modular lattice of level $\ell$ over $K$ were listed.

The contributions of this paper are to generalize the definition of Arakelov-modular lattice to totally real number fields and characterize the existence of Arakelov-modular lattices over various number fields.
More precisely, we give the characterization of existence of Arakelov-modular lattices over the maximal real subfield of a cyclotomic field with prime power degree (Section~\ref{sec:NF-maximalrealprimepower}) and totally real Galois fields with odd degrees (Section~\ref{sec:NF-OddDegree}).
Characterizations of Arakelov-modular lattices of trace type, which are special cases of Arakelov-modular lattices, are given for imaginary quadratic fields (Section~\ref{sec:NF-CM-ImaginaryQF}), totally real quadratic fields (Section~\ref{sec:NF-TotallyReal-QF}) and maximal real subfields of cyclotomic fields with non-prime power degrees (Section~\ref{sec:NF-maximalrealnonprimepower}).
In Section~\ref{sec:NF-ArakeloveModular}, general criteria for the existence of Arakelov-modular lattices over totally real number fields or CM fields are discussed.
The necessary and sufficient conditions for the existence of Arakelov-modular lattices are presented in Section~\ref{sec:NF-CM} for CM fields and in Section~\ref{sec:NF-TotallyReal} for totally real number fields.
Furthermore some examples of lattices constructed by the methods presented are listed.
In particular, two new extremal lattices will be given.

The motivation for constructing modular lattices also comes from coding theory.
Consider a channel with a sender Alice, a legitimate receiver Bob and an eavesdropper Eve.
Such a channel is called a wiretap channel~\cite{Wyner}.
Lattice codes can be used for encoding in wiretap channel and one popular design criterion for wiretap lattice code is secrecy gain, which gives an upper bound on Eve's knowledge of the encoded message~\cite{Oggier}.
In general, it is not easy to compute secrecy gain~\cite{Oggier}.
For a modular lattice, the notion of weak secrecy gain is defined, which is easier to compute and it is conjectured to be the secrecy gain itself~\cite{Oggier}.
This motivates the analysis of the relation between weak secrecy gain and the modularity $\ell$ of a lattice and hence gives incentives to the study of constructing modular lattices (see e.g. \cite{ITW2014}).

%
%

\section{Arakelov-modular Lattices}\label{sec:NF-ArakeloveModular}

Let $K$ be a totally real number field or a CM field with degree $n$ and ring of integers $\Oc_K$.
We consider $K$ to be a Galois extension with Galois group $G=\{\sigma_1=\text{identity},\sigma_2,\dots,\sigma_n\}$.
For $K$ CM, we assume $\sigma_{i+1}$ is the conjugate of $\sigma_i$, $(i=1,3,5,\dots,n-1)$.
An \emph{ideal lattice}~\cite{BayerIdeallattice} over $K$ is a pair $(I,b_\alpha)$, where $I$ is a fractional $\Oc_K-$ideal, $\alpha\in K^\times$ is totally positive (i.e. $\sigma_i(\alpha)>0, 1\leq i\leq n$) and
\begin{eqnarray*}
	b_\alpha:I\times I&\to&\RR\\
	(x,y)&\mapsto&\Tr{\alpha x\bar{y}},
\end{eqnarray*}
is a positive definite symmetric bilinear form.
Here $\rm{Tr}$ is the trace map on $K/\QQ$, the conjugate $-$ is complex conjugation and it is understood to be the identity map when $K$ is totally real.

Note that here we consider the following twisted canonical embedding of $K\hookrightarrow\RR^n$:
\[
x\mapsto(\sqrt{\sigma_1(\alpha)}\sigma_1(x),\dots,\sqrt{\sigma_n(\alpha)}\sigma_n(x))
\]
for $K$ totally real and
\[
	x\mapsto\sqrt{2}(\sqrt{\sigma_1(\alpha)}\re\sigma_1(x),\sqrt{\sigma_2(\alpha)}\im\sigma_2(x),\dots,\sqrt{\sigma_{n-1}(\alpha)}\re\sigma_{n-1}(x),\sqrt{\sigma_n(\alpha)}\im\sigma_n(x))
\]
for $K$ CM.
More specifically, a \emph{generator matrix} for $(I,b_\alpha)$ as a lattice in $\RR^n$ is given by
\[
	\begin{bmatrix}
		\sigma_1(\omega_1)&\dots&\sigma_n(\omega_1)\\
		\vdots&\ddots&\vdots\\
		\sigma_1(\omega_n)&\dots&\sigma_n(\omega_n)
	\end{bmatrix}
	\begin{bmatrix}
		\sqrt{\sigma_1(\alpha)}&\ &\boldsymbol{0}\\
		\ &\ddots&\ \\
		\boldsymbol{0}&\ &\sqrt{\sigma_n(\alpha)}
	\end{bmatrix}
\]
for $K$ totally real and it is given by
\[
	\sqrt{2}\begin{bmatrix}
		\re\sigma_1(\omega_1)&\im\sigma_2(\omega_1)&\dots&\re\sigma_{n-1}(\omega_1)&\im\sigma_n(\omega_1)\\
		\vdots&\vdots&\ddots&\vdots&\vdots\\
		\re\sigma_1(\omega_n)&\im\sigma_2(\omega_n)&\dots&\re\sigma_{n-1}(\omega_n)&\im\sigma_n(\omega_n)
	\end{bmatrix}
	\begin{bmatrix}
		\sqrt{\sigma_1(\alpha)}&\ &\boldsymbol{0}\\
		\ &\ddots&\ \\
		\boldsymbol{0}&\ &\sqrt{\sigma_n(\alpha)}
	\end{bmatrix}
\]
for $K$ CM, where $\{\omega_1,\dots,\omega_n\}$ is a $\ZZ-$basis for $I$.

Let $\D_K$ be the different of $K/\QQ$.
Recall that $\D_K^{-1}=\{x\in K:\Tr{xy}\in\ZZ\ \forall y\in\Oc_K\}$.
Then for an ideal lattice $(I,b_\alpha)$, its dual lattice is given by $(I^*,b_\alpha)$, where $I^*=\alpha^{-1}\D_K^{-1}\bar{I}^{-1}$~\cite{Bayer}.
The definition for Arakelov-modular lattice was given for $K$ being a CM field in~\cite{Bayer}.
Here we extend this definition to totally real number fields.
\begin{definition}\label{def:NF-Arakelovmodular}
Let $K$ be a totally real number field or a CM field and let $\ell$ be a positive integer, an ideal lattice $(I,b_\alpha)$ is said to be \emph{Arakelov-modular of level $\ell$} if there exists $\beta\in K^\times$ such that $I=\beta I^*$ and $\ell=\beta\bar{\beta}$.
\end{definition}
For an Arakelov-modular lattice $(I,b_\alpha)$ of level $\ell$, define $\varphi:I^*\to I$ to be $x\mapsto\beta x$, then $\ell b_\alpha(x,y)=b_\alpha(\varphi(x),\varphi(y))$.
If furthermore, $(I,b_\alpha)$ is an integral lattice, then it is $\ell-$modular.
But this is always true. 
For any fractional ideal $\a$ and prime ideal $\Pf$ in $\Oc_K$, let $v_\Pf(\a)$ denote the exponent of $\Pf$ in the factorization of $\a$, we have
\begin{lemma}\label{lem:NF-integral}
	An Arakelov-modular lattice is integral.
\end{lemma}
\begin{proof}
	Let $(I,b_\alpha)$ be an Arakelov-modular lattice of level $\ell$ and $\beta\in K$ such that $\ell=\beta\bar{\beta}$ and $I=\beta I^*$. 
	Then $I^*=\beta^{-1}I=\alpha^{-1}\D_K^{-1}\bar{I}^{-1}$ gives $\beta=\alpha\D_KI\bar{I}$.
	Take any prime ideal $\Pf$ in $\Oc_K$, we have $v_\Pf(\beta)=v_{\bar{\Pf}}(\beta)$.
	So $v_\Pf(\beta)=\frac{1}{2}v_\Pf(\ell)\geq0$ and hence $\beta\in\Oc_K$.
	Thus $I=\beta I^*\subseteq I^*$ shows $(I,b_\alpha)$ is integral.
\end{proof}
For simplicity, we write $(I,\alpha)$ instead of $(I,b_\alpha)$.
When $\alpha=1$ we say this lattice is of \emph{trace type} \cite{Bayer}.
For any prime integer $p$, let $e_p$ denote its ramification index.
Define 
\begin{eqnarray*}
	\Omega(K)&=&\{p|p\text{ is a prime that ramifies in }K/\QQ\}.\\
	\Omega'(K)&=&\{p|p\in\Omega(K)\text{ and }e_p\text{ is even}\}.
\end{eqnarray*}
We have
\begin{lemma}\label{lem:NF-lram}
If there exists an Arakelov-modular lattice of level $\ell$ over $K$ and $\ell$ is square-free, then $\ell|\prod_{p\in\Omega'(K)}p$.
\end{lemma}
\begin{proof}
If $K$ is totally real, for $\ell$ square-free, this follows from the definition of Arakelov-modular lattices.

If $K$ is CM, Proposition 3.4 of~\cite{Bayer} states that if there exists an Arakelov-modular lattice of level $\ell$ over $K$, then there exists $\lambda\in K$ and a $2^r$th root of unity $\zeta\in K$ for some $r$ such that $\lambda^2=\zeta\ell$.
Since $\ell$ is square-free, the conclusion follows.
\end{proof}
When $K$ has odd degree, for any prime that ramifies in $K$, its ramification index divides the degree of $K$, which is odd, so $\Omega'(K)=\emptyset$.
We have
\begin{corollary}\label{cor:NF-odddegree}
If $K$ has odd degree and there exists an Arakelov-modular lattice of level $\ell$ over $K$, where $\ell$ is square-free, then $\ell=1$.
\end{corollary}

%
%

\section{CM Fields}\label{sec:NF-CM}

We first consider the case when $K$ is CM.
Let $F$ be the maximal totally real subfield of $K$.
For a positive integer $\ell$, write $\ell=\ell_1\ell_2^2$, where $\ell_1$ is square-free.
If there is an Arakelov-modular lattice $(I,\alpha)$ of level $\ell$ over $K$, then the rescaled lattice $(I,\ell^{-1}_2\alpha)$ is an Arakelov-modular lattice of level $\ell_1$~\cite[Proposition 3.2]{Bayer}.
Here we prove the converse result which allows us to restrict to the case when $\ell$ is square-free.
\begin{proposition}
Let $K$ be a CM field and $\ell_1$ be a square-free positive integer.
Assume there is an Arakelov-modular lattice $(I,\alpha)$ of level $\ell_1$.
Take $\ell=\ell_1\ell_2^2$, where $\ell_2$ is a positive integer coprime with $\ell_1$.
Then the rescaled lattice $(\ell_2I,\ell_2^{-1}\alpha)$ is an Arakelov-modular lattice of level $\ell$.
\end{proposition}
\begin{proof}
	Let $(I^*,\alpha)$ be the dual of $(I,\alpha)$, then $(I^*,\ell_2^{-1}\alpha)$ is the dual of $(\ell_2I,\ell_2^{-1}\alpha)$.
	Note that as $\alpha$ is totally positive, $\ell_2^{-1}\alpha$ is also totally positive.
	By the definition of Arakelov-modular lattice, there exists $\beta_1\in K$ such that $\ell_1=\beta_1\bar{\beta}_1$ and $I=\beta_1I^*$.
	Let $\beta=\beta_1\ell_2$, then $\ell=\beta\bar{\beta}$ and $\ell_2I=\ell_2\beta_1I^*=\beta I^*$, which shows $(\ell_2I,\ell_2^{-1}\alpha)$ is an Arakelov-modular lattice of level $\ell$. 
\end{proof}

From now on we consider $\ell$ to be square-free.
Moreover, we note that
\begin{proposition}\label{prop:NF-vPeven}
	There exists an Arakelov-modular lattice of trace type over $K$ if and only if there exists $\alpha\in K^\times$ totally positive, $\beta\in K^\times$ such that $\ell=\beta\bar{\beta}$ and $v_\Pf(\alpha^{-1}\beta\D_K^{-1})$ is even whenever $\Pf=\bar{\Pf}$ for any prime ideal $\Pf$ in $\Oc_K$.
\end{proposition}
\begin{proof}
	By the definition of Arakelov-modular lattice, there exists an Arakelov-modular lattice, say $(I,\alpha)$, if and only if $\alpha$ is totally positive, $\exists\beta\in K^\times$ such that $\ell=\beta\bar{\beta}$ and the decomposition $I\bar{I}=\alpha^{-1}\beta\D_K^{-1}$ is possible, which is equivalent to requiring $v_\Pf(\alpha^{-1}\beta\D_K^{-1})$ to be even whenever $\Pf=\bar{\Pf}$ for all prime ideals $\Pf$ in $\Oc_K$.
\end{proof}

%
%

\subsection{Imaginary Quadratic Number Fields}\label{sec:NF-CM-ImaginaryQF}

Suppose $K=\QQ(\sqrt{-d})$, where $d$ is a square-free positive integer.
For any prime $p$ that ramifies in $K/\QQ$, $p$ is totally ramified with a unique $\Oc_K$ prime ideal above it, which we denote by $\Pf_p$.
\begin{proposition}
	There exists an Arakelov-modular lattice of level $\ell$ of trace type over $K$ if and only if $\ell=d$.
	Moreover, $(I,1)$, where
	\[
		I=\begin{cases}
			\Pf_2^{-1}&d\md{1,2}{4}\\
			\Oc_K&d\md{3}{4}
		\end{cases},
	\]
	is an Arakelov-modular lattice of level $d$.
\end{proposition}
\begin{proof}
	If $d\md{1,2}{4}$, $\Omega'(K)=\Omega(K)=\{p:p|2d\}$, $\D_K=(2\sqrt{-d})$.
	If $d\md{3}{4}$, $\Omega'(K)=\Omega(K)=\{p:p|d\}$, $\D_K=(\sqrt{-d})$.
	For any $p\in\Omega(K)$,
\[
	v_{\Pf_p}(\D_K)=\begin{cases}
		1&p\text{ odd}\\
		2&p=2,d\md{1}{4}\\
		3&p=2,d\md{2}{4}.
	\end{cases}
\]
Suppose there exists an Arakelov-modular lattice of level $\ell$.
By Lemma \ref{lem:NF-lram}, we only need to consider $\ell$ being a divisor of $\prod_{p\in\Omega(K)}p$.
Take $\beta$ such that $\ell=\beta\bar{\beta}$, by the proof of Lemma \ref{lem:NF-integral},
\[
	v_{\Pf_p}(\beta)=\frac{1}{2}v_{\Pf_p}(\ell)=\begin{cases}
1&p|\ell\\
0&p\nmid\ell
	\end{cases}.
\]
By Proposition \ref{prop:NF-vPeven} and the above discussion,
\[
	\ell=\begin{cases}
		\displaystyle\prod_{p\text{ odd },p|d}&d\md{1,3}{4}\\
		2\displaystyle\prod_{p\text{ odd },p|d}&d\md{2}{4}
	\end{cases}\ =d.
\]
On the other hand, take $\ell=d$.
Then $\beta=\sqrt{-d}$ satisfies $\ell=\beta\bar{\beta}$ and 
\[
	\beta\D_K^{-1}=
	\begin{cases}
		\frac{1}{2}\Oc_K&d\md{1,2}{4}\\
		\Oc_K&d\md{3}{4}.
	\end{cases}
\]
Take 	\[
		I=\begin{cases}
			\Pf_2^{-1}&d\md{1,2}{4}\\
			\Oc_K&d\md{3}{4}
		\end{cases},
	\]
then $I\bar{I}=\beta\D_K^{-1}$ shows $(I,1)$ is an Arakelov-modular lattice of level $d$. 
\end{proof}
Furthermore, we have the following observations:
\begin{itemize}

\item[1.] For $d\md{3}{4}$, $(\Oc_K,1)$ has generator matrix $M$ and Gram matrix $G$ given by
\[
	M=\sqrt{2}\begin{bmatrix}
1&0\\
\frac{1}{2}&-\frac{\sqrt{d}}{2}
	\end{bmatrix},\ 
	G=\begin{bmatrix}
2&1\\
1&\frac{d+1}{2}
	\end{bmatrix}.
\]
Hence $(\Oc_K,1)$ is an even $d-$modular lattice of dimension $2$ with minimum $2$. 

\item[2.] For $d\md{2}{4}$, $\{1,\frac{\sqrt{-d}}{2}\}$ is a basis for $\Pf_2^{-1}$.
Then $(\Pf_2^{-1},1)$ has generator matrix $M$ and Gram matrix $G$ given by
\[
	M=\sqrt{2}\begin{bmatrix}
	1&0\\
	0&-\frac{\sqrt{d}}{2}
	\end{bmatrix},\ 
	G=\begin{bmatrix}
2&0\\
0&\frac{d}{2}
	\end{bmatrix},
\]
which shows $(\Pf_2^{-1},1)$ is an odd $d-$modular lattice of dimension $2$ with minimum $2$.
\item[3.] For $d\md{1}{4}$, $\{1,\frac{1+\sqrt{-d}}{2}\}$ is a basis for $\Pf_2^{-1}$.
$(\Pf_2^{-1},1)$ has generator matrix $M$ and Gram matrix $G$ given by
\[
	M=\sqrt{2}\begin{bmatrix}
1&0\\
\frac{1}{2}&-\frac{\sqrt{d}}{2}
	\end{bmatrix},\ 
	G=\begin{bmatrix}
2&1\\
1&\frac{d+1}{2}
	\end{bmatrix}.
\]
Hence $(\Pf_2^{-1},1)$ is an odd $d-$modular lattice of dimension $2$ with minimum $2$.
\end{itemize}

%
%

\section{Totally Real Number Fields}\label{sec:NF-TotallyReal}

In this section we consider the case when $K$ is a totally real number field.
Let $\ell$ be a positive integer and write $\ell=\ell_1\ell_2^2$, where $\ell_1$ is square-free.
We have the following.
\begin{proposition}
There exists an Arakelov-modular lattice of level $\ell$ over $K$ if and only if there exists an Arakelov-modular lattice of level $\ell_1$ over $K$.
\end{proposition}
\begin{proof}
	First, let $(I,\alpha)$ be an Arakelov-modular lattice of level $\ell$ over $K$, $(I^*,\alpha)$ be the dual lattice.
	Then $(\ell_2I^*,\ell_2^{-1}\alpha)$ is the dual lattice of $(I,\ell_2^{-1}\alpha)$.
	Take $\beta\in K^\times$ such that $I=\beta I^*$, $\ell=\beta^2$.
	Let $\beta_1=\frac{\beta}{\ell_2}\in K^\times$, then $\ell_1=\beta_1^2$ and $I=\beta_1\ell_2I^*$, which shows $(I,\ell_2^{-1}\alpha)$ is an Arakelov-modular lattice of level $\ell_1$.

	Conversely, let $(I,\alpha)$ be an Arakelov-modular lattice of level $\ell_1$ over $K$ and let $(I^*,\alpha)$ be the dual lattice.
	Then $(I^*,\ell_2^{-1}\alpha)$ is the dual of $(\ell_2I,\ell_2^{-1}\alpha)$.
	Take $\beta_1\in K^\times$ such that $\ell_1=\beta_1^2$ and $I=\beta_1 I^*$.
	Let $\beta=\beta_1\ell_2$, then $\ell=\beta^2$ and $\ell_2I=\ell_2\beta_1 I^*=\beta I^*$, which shows $(\ell_2I,\ell_2^{-1}\alpha)$ is an Arakelov-modular lattice of level $\ell$.
\end{proof}
From now on we consider $\ell$ to be square-free.
\begin{proposition}\label{prop:NF-vPeventotalreal}
	There exists an Arakelov-modular lattice over $K$ if and only if there exists $\alpha\in K^\times$ totally positive, $\beta\in K^\times$ such that $\ell=\beta^2$ and $v_\Pf(\alpha^{-1}\beta\D_K^{-1})$ is even for any prime ideal $\Pf$ in $\Oc_K$.
\end{proposition}
\begin{proof}
	By the definition of Arakelov-modular lattice, there exists an Arakelov-modular lattice, say $(I,\alpha)$, if and only if $\alpha\in K^\times$ is totally positive, $\exists\beta\in K^\times$ such that $\ell=\beta^2$ and the decomposition $I\bar{I}=I^2=\alpha^{-1}\beta\D_K^{-1}$ is possible, which is equivalent to requiring $v_\Pf(\alpha^{-1}\beta\D_K^{-1})$ to be even for all prime ideals $\Pf$ in $\Oc_K$.
\end{proof}

%
%

\subsection{Totally Real Quadratic Fields}\label{sec:NF-TotallyReal-QF}
Let $K=\QQ(\sqrt{d})$, where $d$ is a square-free positive integer.
For any prime $p$ that ramifies in $K/\QQ$, $p$ is totally ramified with a unique $\Oc_K$ prime ideal above it, which we denote by $\Pf_p$.
\begin{proposition}
	There exists an Arakelov-modular lattice of level $\ell$ of trace type over $K$ if and only if $\ell=d$.
	Moreover, $(I,1)$, where
	\[
		I=\begin{cases}
			\Pf_2^{-1}&d\md{2,3}{4}\\
			\Oc_K&d\md{1}{4}
		\end{cases},
	\]
	is an Arakelov-modular lattice of level $d$.
\end{proposition}
\begin{proof}
	If $d\md{2,3}{4}$, $\Omega'(K)=\Omega(K)=\{p:p|2d\}$, $\D_K=(2\sqrt{d})$.
	If $d\md{1}{4}$, $\Omega'(K)=\Omega(K)=\{p:p|d\}$, $\D_K=(\sqrt{d})$.
	For any $p\in\Omega(K)$,
\[
	v_{\Pf_p}(\D_K)=\begin{cases}
		1&p\text{ odd}\\
		2&p=2,d\md{3}{4}\\
		3&p=2,d\md{2}{4}.
	\end{cases}
\]
Suppose there exists an Arakelov-modular lattice of level $\ell$.
By Lemma \ref{lem:NF-lram}, we only need to consider $\ell$ being a divisor of $\prod_{p\in\Omega(K)}p$.
Take $\beta$ such that $\ell=\beta^2$, then
\[
	v_{\Pf_p}(\beta)=\frac{1}{2}v_{\Pf_p}(\ell)=\begin{cases}
1&p|\ell\\
0&p\nmid\ell
	\end{cases}.
\]
By Proposition \ref{prop:NF-vPeventotalreal} and the above discussion,
\[
	\ell=\begin{cases}
		\displaystyle\prod_{p\text{ odd },p|d}&d\md{1,3}{4}\\
		2\displaystyle\prod_{p\text{ odd },p|d}&d\md{2}{4}
	\end{cases}\ =d.
\]
On the other hand, take $\ell=d$.
Then $\beta=\sqrt{d}$ satisfies $\ell=\beta^2$ and 
\[
	\beta\D_K^{-1}=
	\begin{cases}
		\frac{1}{2}\Oc_K&d\md{2,3}{4}\\
		\Oc_K&d\md{1}{4}.
	\end{cases}
\]
Take 	\[
		I=\begin{cases}
			\Pf_2^{-1}&d\md{2,3}{4}\\
			\Oc_K&d\md{1}{4}
		\end{cases},
	\]
then $I^2=\beta\D_K^{-1}$ shows $(I,1)$ is an Arakelov-modular lattice of level $d$. 
\end{proof}
Furthermore, we have the following observations:
\begin{itemize}

\item[1.] For $d\md{1}{4}$, $(\Oc_K,1)$ has generator matrix $M$ and Gram matrix $G$ given by
\[
	M=\begin{bmatrix}
1&1\\
\frac{1+\sqrt{d}}{2}&\frac{1-\sqrt{d}}{2}
	\end{bmatrix},\ 
	G=\begin{bmatrix}
2&1\\
1&\frac{d+1}{2}
	\end{bmatrix}.
\]
Hence $(\Oc_K,1)$ is an odd $d-$modular lattice of dimension $2$ with minimum $2$. 

\item[2.] For $d\md{2}{4}$, $\{1,\frac{\sqrt{d}}{2}\}$ is a basis for $\Pf_2^{-1}$.
Then $(\Pf_2^{-1},1)$ has generator matrix $M$ and Gram matrix $G$ given by
\[
	M=\begin{bmatrix}
	1&1\\
	\frac{\sqrt{d}}{2}&-\frac{\sqrt{d}}{2}
	\end{bmatrix},\ 
	G=\begin{bmatrix}
2&0\\
0&\frac{d}{2}
	\end{bmatrix},
\]
which shows $(\Pf_2^{-1},1)$ is an odd $d-$modular lattice of dimension $2$ with minimum $1$ for $d=2$ and minimum $2$ otherwise.
\item[3.] For $d\md{3}{4}$, $\{1,\frac{1+\sqrt{d}}{2}\}$ is a basis for $\Pf_2^{-1}$.
$(\Pf_2^{-1},1)$ has generator matrix $M$ and Gram matrix $G$ given by
\[
	M=\begin{bmatrix}
1&1\\
\frac{1+\sqrt{d}}{2}&\frac{1-\sqrt{d}}{2}
	\end{bmatrix},\ 
	G=\begin{bmatrix}
2&1\\
1&\frac{d+1}{2}
	\end{bmatrix}.
\]
Hence $(\Pf_2^{-1},1)$ is an even $d-$modular lattice of dimension $2$ with minimum $2$.
\end{itemize}

%
%

\subsection{Maximal Real Subfield of a Cyclotomic Field -- The Prime Power Case}\label{sec:NF-maximalrealprimepower}

Let $p$ be an odd prime, $r$ a positive integer and $\zeta_{p^r}$ a primitive $p^r$th root of unity.
In this subsection we consider the case $K=\QQ(\zeta_{p^r}+\zeta_{p^r}^{-1})$.
Let Mod$_T(p^r)$ denote the set of $\ell$ such that there exists an Arakelov-modular lattice of trace type of level $\ell$ over $K=\QQ(\zeta_{p^r}+\zeta_{p^r}^{-1})$ and let Mod$(p^r)$ denote the set of $\ell$ such that there exists an Arakelov-modular lattice of level $\ell$ over $K=\QQ(\zeta_{p^r}+\zeta_{p^r}^{-1})$.

Recall that $p$ is the only prime that ramifies and it is totally ramified with ramification index $\frac{p^r(p-1)}{2}$.
From Lemma \ref{lem:NF-lram} we have
\begin{corollary}\label{cor:NF-MOD}
	\begin{itemize}
		\item[1.] $\text{Mod}_T(p^r)\subseteq\text{Mod}(p^r)\subseteq\{1\}\text{ if }p\md{3}{4}$,
		\item[2.] $\text{Mod}_T(p^r)\subseteq\text{Mod}(p^r)\subseteq\{1,p\}\text{ if }p\md{1}{4}$.
	\end{itemize}
\end{corollary}
Let $\Pf$ denote the prime ideal in $\Oc_K$ above $p$, then~\cite{Swinnerton-Dyer}
\begin{equation}\label{eqn:NF-vBD}
	v_\Pf(\D_K)=\frac{1}{2}(p^{r-1}(pr-r-1)-1)\equiv
	\begin{cases}
		1~{\rm mod}~2&p\md{1}{4}\\
		0~{\rm mod}~2&p\md{3}{4}
	\end{cases}
\end{equation}
We have the following characterization of Arakelov-modular lattices of trace type.
For the more general case, characterizations of Arakelov-modular lattices will be given in Proposition~\ref{prop:NF-MaximalPrimeClassification}.
\begin{proposition}\label{prop:NF-ModMaximalPrime}
	There exists an Arakelov-modular lattice of level $\ell$ of trace type over $\QQ(\zeta_{p^r}+\zeta_{p^r}^{-1})$ if and only if $\ell\in\text{Mod}_T(p^r)$, where Mod$_T(p^r)$ is given by
	\begin{itemize}
		\item[1.] Mod$_T(p^r)=\{1\}$, if $p\md{3}{4}$;
		\item[2.] Mod$_T(p^r)=\emptyset$, if $p\md{1}{8}$;
		\item[3.] Mod$_T(p^r)=\{p\}$, if $p\md{5}{8}$;
	\end{itemize}
\end{proposition}
\begin{proof}
	1. If $p\md{3}{4}$, $1$ is a square in $K$ with $1=1^2$.
	By (\ref{eqn:NF-vBD}), $v_\Pf(\D_K^{-1})$ is even.
	Then by Proposition \ref{prop:NF-vPeventotalreal}, $1\in\text{Mod}_T(p^r)$.
	By Corollary \ref{cor:NF-MOD}, Mod$_T(p^r)=\{1\}$.

	2. If $p\md{1}{8}$, $1$ is a square in $K$ with $1=1^2$.
	By (\ref{eqn:NF-vBD}), $v_\Pf(\D_K^{-1})$ is odd.
	Then by Proposition \ref{prop:NF-vPeventotalreal}, $1\notin\text{Mod}_T(p^r)$.
	Now take $\ell=p$, we have $\sqrt{p}\in\QQ(\sqrt{p})\subseteq\QQ(\zeta_p)\cap\RR\subseteq\QQ(\zeta_{p^r})\cap\RR=K$ (see~\cite{Washington} p.17).
	Let $\beta=\sqrt{p}$, then $\ell=\beta^2$ and $v_\Pf(\beta)=\frac{1}{4}p^{r-1}(p-1)$ is even.
	Hence $v_\Pf(\beta\D_K^{-1})$ is odd and by Proposition \ref{prop:NF-vPeventotalreal}, $p\notin\text{Mod}_T(p^r)$.
	By Corollary \ref{cor:NF-MOD}, Mod$_T(p^r)=\emptyset$.

	3. If $p\md{5}{8}$, same as above $1\notin\text{Mod}_T(p^r)$.
	Take $\ell=p$.
	Similarly, let $\beta=\sqrt{p}$, then $\ell=\beta^2$ and $v_\Pf(\beta)=\frac{1}{4}p^{r-1}(p-1)$ is odd.
	Hence $v_\Pf(\beta\D_K^{-1})$ is even and by Proposition \ref{prop:NF-vPeventotalreal}, $p\in\text{Mod}_T(p^r)$.
	By Corollary \ref{cor:NF-MOD}, Mod$_T(p^r)=\{p\}$.
\end{proof}
Define
\begin{eqnarray*}
	s_1&:=&v_\Pf(\D_K^{-1}),\\
	s_2&:=&\frac{1}{2}v_\Pf(p)=\frac{1}{4}p^{r-1}(p-1)
\end{eqnarray*}
From the above proof we have
\begin{corollary}
	If $p\md{3}{4}$, $(\Pf^{\frac{s_1}{2}},1)$ is an unimodular lattice over $K$.
	If $p\md{5}{8}$, $(\Pf^{\frac{s_1+s_2}{2},1})$ is an Arakelov-modular lattice of level $p$ over $K$.
\end{corollary}
\begin{lemma}\label{lem:NF-totalpos}
	$\Pf=(2-2\cos\frac{2\pi}{p^r})$ and $2-2\cos\frac{2\pi}{p^r}$ is totally positive in $K$.
\end{lemma}
\begin{proof}
	Since $\sigma(\cos\frac{2\pi}{p^r})<1$ for all $\sigma\in$Gal$(K/\QQ)$, $2-2\cos\frac{2\pi}{p^r}$ is totally positive in $K$.
	Moreover,
	\[
		(1-\zeta_{p^r})(1-\zeta_{p^r}^{-1})=2-2\cos\frac{2\pi}{p^r}
	\]
	generates $\Pf$.
\end{proof}
\begin{proposition}\label{prop:NF-MaximalPrimeClassification}
	There exists an Arakelov-modular lattice of level $\ell$ over $\QQ(\zeta_{p^r}+\zeta_{p^r}^{-1})$ if and only if $\ell\in\text{Mod}(p^r)$, where Mod$(p^r)$ is given by
	\begin{itemize}
		\item[1.] Mod$(p^r)=\{1,p\}$, if $p\md{1}{4}$;
		\item[2.] Mod$(p^r)=\{1\}$, if $p\md{3}{4}$.
	\end{itemize}
\end{proposition}
\begin{proof}
	1. Take $p\md{1}{4}$, $\ell=1$.
	Let $\alpha=(2-2\cos\frac{2\pi}{p^r})^{-1}$, by Lemma \ref{lem:NF-totalpos}, $\alpha$ is totally positive.
	Moreover,
	\[
		v_\Pf(\alpha^{-1}\D_K^{-1})=1+s_1
	\]
	is even. 
	By Proposition \ref{prop:NF-vPeventotalreal}, $1\in\text{Mod}(p^r)$.
	
	2. Take $p\md{1}{8}$, $\ell=p$.
	Let $\alpha=(2-2\cos\frac{2\pi}{p^r})^{-1}$ and $\beta=\sqrt{p}$, then $\ell=\beta^2$ and
	\[
		v_\Pf(\alpha^{-1}\beta\D_K^{-1})=1+s_1+s_2
	\]
	is even. 
	By Proposition \ref{prop:NF-vPeventotalreal}, $p\in\text{Mod}(p^r)$.

	By Corollary \ref{cor:NF-MOD} and Proposition \ref{prop:NF-ModMaximalPrime}, the proof is completed.
\end{proof}
\begin{example}
Take $K=\QQ(\zeta_{13}+\zeta_{13}^{-1})$, by Proposition~\ref{prop:NF-MaximalPrimeClassification} there exists a $6-$dimensional unimodular lattice over $K$.
The ideal lattice $(\Pf_{13}^{-3},(2-2\cos\frac{2\pi}{13})^{-1})$ gives us such a lattice, where $\Pf_{13}$ is the unique prime ideal in $\Oc_K$ above $13$.
This lattice is isometric to the lattice $\ZZ^6$.
\end{example}

%
%

\subsection{Maximal Real Subfield of a Cyclotomic Field -- The Non-Prime Power Case}\label{sec:NF-maximalrealnonprimepower}
Let $n\not\equiv 2~{\rm mod}~4$ be an integer which is not a prime power.
Set $L=\QQ(\zeta_n)$ and $K=\QQ(\zeta_n+\zeta_n^{-1})$. 
For any $p\in\QQ$ a prime dividing $n$, write $n=p^{r_p}n'_p$ with $(p,n'_p)=1$. 
Let $\Pf_p$ be the prime ideal in $\Oc_K$ above $p$. 
We have $v_{\Pf_p}(p)=p^{r_p-1}(p-1)$ and $v_{\Pf_p}(\D_K)=p^{r_p-1}(pr_p-r_p-1)$.

For any divisor $d$ of $n$, define 
\[
	d_{mod3}=\prod_{\substack{p|d \\ p\equiv3~{\rm mod}~4}} p, \qquad\qquad d_{mod1} = \prod_{\substack{p|d\\ p\equiv1~{\rm mod}~4}}p,\qquad\qquad \tilde{n}=\prod_{\substack{p\in\Omega(K)\\p\neq2}}p.
\]
\begin{lem} \label{lem:NF-square}
\begin{itemize}
\item[1.] For any $d|n$, $d_{mod1}$ is always a square in $K$.
\item[2.] For any $d|n$, $d_{mod3}$ is a square in $K$ if and only if one of the following conditions is satisfied:
         \begin{itemize}
         \item $n$ is even
         \item $n$ is odd and $d_{mod3}$ has an even number of distinct prime factors 
         \end{itemize}
\item[3.] For $n$ even, $2$ is a square in $K$ if and only if $n\equiv 0~{\rm mod}~8$.
\end{itemize}
\end{lem}
\begin{proof}
	For any odd prime $p|n$, $\QQ(\zeta_p)\subseteq L$ and \cite[p.17]{Washington}
\[
\sqrt{p}\in\QQ(\zeta_p)\iff p\equiv 1~{\rm mod}~4,\qquad \sqrt{-p}\in\QQ(\zeta_p)\iff p\equiv3~{\rm mod}~4.
\]
Part 1 follows immediately. 

To prove 2, assume $p\equiv3~{\rm mod}~4$, then $\sqrt{-p}\in L$.

If $n$ is even, $i\in L$, $\sqrt{p}=\frac{\sqrt{-p}}{i}\in L\cap\RR=K$.

If $n$ is odd and $d_{mod3}$ has an even number of distinct prime factors,
\[
\sqrt{d_{mod3}} = \prod_{\substack{p|d\\p\equiv3~{\rm mod}~4}}\sqrt{p} = \prod_{\substack{p|d\\p\equiv3~{\rm mod}~4}}\sqrt{-p}\in L\cap\RR=K.
\]

On the other hand, assume $n$ is odd, and $d_{mod3}$, a square in $K$, has an odd number of distinct prime factors. 
Let $p_0$ be any prime factor of $d_{mod3}$. We have $\sqrt{-p_0}\in L$ and
\[
\sqrt{p_0}=\frac{\sqrt{d_{mod3}}}{\prod_{\substack{p|d_{mod3}\\ p\neq p_0}}\sqrt{p}} = \frac{\sqrt{d_{mod3}}}{\prod_{\substack{p|d_{mod3}\\ p\neq p_0}} \sqrt{-p}} \in L.
\]
So $i=\frac{\sqrt{-p_0}}{\sqrt{p_0}}\in L$, which implies $4|n$, a contradiction.

Now consider $n$ even. If $8|n$, $\sqrt{2}\in\QQ(\zeta_8)\cap\RR\subseteq L\cap\RR=K$.

Conversely, if $\sqrt{2}\in K\subseteq L$, since $i\in L$, we have $\zeta_8\in K$ and hence $8|n$.
\end{proof}

Let Mod$_T(n)$ denote the set of $\ell$ such that there exists an Arakelov-modular lattice of trace type of level $\ell$ over $K$.

\begin{lem}\label{lem:NF-primefactor}
Mod$_T(n)\neq\emptyset$ if and only if
\begin{itemize}
	\item[1.] $n_{mod1}=1$;
	\item[2.] For any $\ell\in Mod_T(n)$, $\ell=\tilde{n}$ or $2\tilde{n}$ and $\ell|n$;
	\item[3.] $\ell$ is a square in $K$.
\end{itemize}
\end{lem}
\begin{proof}
	First we assume Mod$_T(n)\neq\emptyset$.
	For any $p\in\Omega(K)$, $p\neq2$, $v_{\Pf_p}(\D_K^{-1})=-p^{r_p-1}(pr_p-r_p-1)$ is odd.
	By Proposition \ref{prop:NF-vPeventotalreal}, $\ell$ is a square, $p|\ell$ and $\frac{1}{2}v_{\Pf_p}(\ell)=\frac{1}{2}p^{r_p-1}(p-1)$ must be odd.
	The proof then follows from Lemma \ref{lem:NF-lram}. 

	Conversely, assume all three conditions are satisfied.
	By Proposition \ref{prop:NF-vPeventotalreal}, it suffices to prove $v_{\Pf_p}(\beta\D_K^{-1})$ is even for all prime ideal $\Pf_p$, where $v_{\Pf_p}(\beta)=\frac{1}{2}v_{\Pf_p}(\ell)$.
	Conditions 1 and 2 ensure that $v_{\Pf_p}(\beta\D_K^{-1})$ is even for all prime ideal $\Pf_p$.
\end{proof}
The above discussion gives the characterization of the existence of Arakelov-modular lattices of trace type over $K$, the maximal real subfield of the cyclotomic field generated by a primitive $n$th root of unity:
\begin{proposition}\label{prop:NF-totalrealnonprime}
Let $n\not\equiv2~{\rm mod}~4$ be a positive integer which is not a prime power.

If $n$ has any prime factor $p\equiv1~{\rm mod}~4$, Mod$_T(n)=\emptyset$. 

Otherwise,
\begin{enumerate}
\item If $n$ is odd and has an even number of distinct prime factors, Mod$_T(n)=\{\tilde{n}\}$; 
\item If $n$ is odd and has an odd number of distinct prime factors, Mod$_T(n)=\emptyset$; 
\item If $n=4m$, where $m$ is odd, Mod$_T(n)=\{\tilde{n}\}$; 
\item If $n=2^rm$, where $r\geq3$ and $m$ is odd, Mod$_T(n)=\{\tilde{n},2\tilde{n}\}$. 
\end{enumerate}  
\end{proposition}
Recall that the {\em minimum} of a lattice, $\mu_L$, is defined to be
\[
	\mu_L:=\min\{b(x,x):x\in L,x\neq0\}.
\]
For unimodular lattices and even $\ell-$modular lattice, where $\ell\in\{2,3,5,7,11,23\}$, the upper bounds for this  minimum are obtained in~\cite{Quebbemann,Conway1990,Rains}.
A unimodular or an even $\ell-$modular lattice for $\ell\in\{2,3,5,7,11,23\}$ that achieves this upper bound is called {\em extremal}.
For a detailed list of extremal lattices we refer the reader to the on-line lattice catalogue \cite{LatticeWeb}.
In Table~\ref{tb:NF-examples}, we list a few constructions of existing lattices, note that the first two lattices are extremal.
\begin{table}[ht!]
\begin{adjustbox}{center}
\small
\begin{tabular}{|c|c|c|c|c|c|c|c|} 
\hline	
$\ell$ &                $K$                &             $I$            &          $\alpha$                    &  Dim  & $\min$ &  NAME   \\                     
\hline
 $7$   & $\QQ(\zeta_{28}+\zeta_{28}^{-1})$ & $\Pf_{7}^{-1}\Pf_2^{-1}$   &                 $1$                  &  $6$  &   $2$  &  $A6$\^{}$(2)$ \\
 $11$  & $\QQ(\zeta_{44}+\zeta_{44}^{-1})$ & $\Pf_{11}^{-2}\Pf_2^{-1}$  &                 $1$                  &  $10$ &   $6$  &  $A10$\^{}$(3)$ \\          
 $23$  & $\QQ(\zeta_{92}+\zeta_{92}^{-1})$ & $\Pf_{23}^{-5}\Pf_2^{-1}$  &                 $1$                  &  $22$ &   $12$ &  $A22$\^{}$(6)$ \\           
\hline
\end{tabular}
\end{adjustbox}
\caption{Examples of lattices $(I,\alpha)$ obtained from $K$ such that $(I,\alpha)$ is an Dim$-$dimensional Arakelov-modular lattice of level $\ell$ with minimum $\min$ and isometric to the existing lattice NAME from \cite{LatticeWeb}. Here $\Pf_p$ denotes the unique prime ideal in $\Oc_K$ above $p$.}
\label{tb:NF-examples}
\end{table}
Furthermore, we give a new extremal lattice in Example~\ref{ex:NF-new3mod}.
\begin{example}\label{ex:NF-new3mod}
[\textbf{New extremal $3-$modular lattice}]
Take $K=\QQ(\zeta_{36}+\zeta_{36}^{-1})$, by Proposition~\ref{prop:NF-totalrealnonprime} there exists a $6-$dimensional $3-$modular lattice over $K$.
The ideal lattice $(\Pf_3^{-3}\Pf_2^{-1},1)$ gives us such a lattice with minimum $2$, where $\Pf_3$ (resp. $\Pf_2$) is the unique prime ideal in $\Oc_K$ above $3$ (resp. $2$).
This lattice is an extremal $6-$dimensional $3-$modular lattice.
\end{example}

%
%

\section{Number Fields with Odd Degree}\label{sec:NF-OddDegree}

Let $K$ be a Galois extension with odd degree $n$.
We consider $K$ to be either totally real or CM and take $\ell$ to be a positive square-free integer.
Let Mod$(K)$ denote the set of $\ell$ such that there exists an Arakelov-modular lattice of level $\ell$ over $K$.
By Corollary~\ref{cor:NF-odddegree}, Mod$(K)\subseteq\{1\}$.
Furthermore, by Propositions~\ref{prop:NF-vPeven} and \ref{prop:NF-vPeventotalreal}, we have
\begin{proposition}\label{prop:Nf-OddDegree}
	Mod$(K)\neq\emptyset$ if and only if Mod$(K)=\{1\}$ and there exists $\alpha\in K^\times$ totally positive such that for any prime ideal $\Pf$
	\begin{itemize}
		\item[1.] $v_\Pf(\alpha^{-1}\D_K^{-1})$ is even whenever $\Pf=\bar{\Pf}$ for $K$ is CM; and
		\item[2.] $v_\Pf(\alpha^{-1}\D_K^{-1})$ is even for all $\Pf$ when $K$ is totally real.
	\end{itemize}
\end{proposition}
\begin{proposition}\label{prop:NF-OddDegreeUniMod}
Let $K$ be a Galois field with odd degree.
We assume $K$ to be either totally real or CM.
Then Mod$(K)=\{1\}$.
\end{proposition}
\begin{proof}
We claim that for any $\Pf$ a prime ideal in $\Oc_K$, $v_\Pf(\D_K)$ is even.
By Proposition~\ref{prop:Nf-OddDegree}, taking $\alpha=1$, the proof is completed.

\textit{Proof of claim:} Fix $\Pf$ a prime ideal in $\Oc_K$, take $p$ such that $p\ZZ=\Pf\cap\ZZ$.
Suppose $p$ has ramification index $e$ and inertia degree $f$.
Let $\QQ_p$, $K_\Pf$ be the completion of $\QQ$ (resp. $K$) with respect to the $p-$adic valuation (resp. $\Pf-$adic valuation).
Then $K_\Pf$ is a Galois extension of $\QQ_p$ with degree $ef$ \cite[p.103]{Serre}.
Let $G(K_\Pf|\QQ_p)$ be the Galois group of $K_\Pf/\QQ_p$.
For $i\geq0$, define \cite[p.61]{Serre}
\[
	G_i:=\{\sigma\in G(K_\Pf|\QQ_p)|v_\Pf(\sigma(a)-a)\geq i+1\ \forall a\in\Oc_{K_\Pf}\}.
\]
Let $\D_{K_\Pf/\QQ_p}$ be the different of $K_\Pf/\QQ_p$, then \cite[p.64]{Serre}
\[
	v_\Pf(\D_{K_\Pf/\QQ_p})=\sum_{i=0}^\infty(|G_i|-1).
\]
Since each $G_i$ is a normal subgroup of $G(K_\Pf|\QQ_p)$ \cite[p.62]{Serre}, which has odd cardinality, $v_\Pf(\D_{K_\Pf/\QQ_p})$ is even.
As \cite[p.196]{Neukirch}
\[
	\D_{K/\QQ}=\prod_\Pf\D_{K_\Pf/\QQ_p},
\]
we have $v_\Pf(\D_{K/\QQ})$ is even.
\end{proof}
\begin{remark}\label{rem:NF-OddDegreeUniMod}
	By the above proof, for a Galois field $K$ with odd degree which is either totally real or CM, $(\D_K^{-\frac{1}{2}},1)$ is an Arakelov-modular lattice of level $1$, in particular, it is a unimodular lattice.
\end{remark}
Next we give an example of unimodular lattice constructed from a totally real Galois field of odd degree.
To our best knowledge, it is a new lattice.
\begin{example}\label{ex:NF-newunimod}
[\textbf{New extremal unimodular lattice}]
Take $K$ to be the unique $21$ dimensional subfield of the cyclotomic field $\QQ(\zeta_{49})$.
By Proposition~\ref{prop:NF-OddDegreeUniMod}, there exists a unimodular lattice over $K$.
By Remark~\ref{rem:NF-OddDegreeUniMod}, $(\D_K^{-\frac{1}{2}},1)$ is such a lattice.
Using Magma~\cite{Magma}, we get this lattice has minimum $2$, hence it is an extremal $21-$dimensional unimodular lattice.
\end{example}

%
%

\section*{Acknowledgments}
The author would like to thank Prof. Fr\'ed\'erique Oggier and Prof. Christian Maire for the useful discussions.
This work is supported by Nanyang President Graduate Scholarship.

%
%
%

\end{document}